\documentclass[a4paper,twoside,english]{amsart}
\usepackage{amsmath,amssymb,amsthm,amscd} 
\usepackage{amssymb,fge} 
\usepackage{tikz-cd}
\usepackage{mathtools}
\usepackage{comment}
\usepackage{graphicx}
\usepackage[left=3.2cm,right=3.5cm,top=3cm,bottom=3cm]{geometry}
\usepackage{indentfirst}           
\usepackage{underscore}
\usepackage{enumitem}
\usepackage{relsize}
\usepackage{varwidth}
\usepackage{mathtools} 
\usepackage{hyperref}
\hypersetup{
    colorlinks=true,
    linkcolor=blue,
    filecolor=magenta,      
    urlcolor=cyan,
}
\usepackage{amssymb}
\usepackage[all]{xy}
\usepackage{todonotes}

\newtheorem{thm}{Theorem}
\newtheorem{cor}[thm]{Corollary}
\newtheorem{prop}[thm]{Proposition}
\newtheorem{lem}[thm]{Lemma}

\theoremstyle{definition} 

\newtheorem{rem}[thm]{Remark}

\theoremstyle{remark}

\numberwithin{thm}{section}
\numberwithin{equation}{thm}
\newcommand{\be}{\begin{equation}}
\newcommand{\ee}{\end{equation}}

\newcommand{\Mod}[1]{\ (\mathrm{mod}\ #1)}

\DeclareMathOperator{\Gal}{Gal}

\usepackage[utf8]{inputenc}

\title{On the factorization of iterates of $x^d+c$ in large degree}

\author{Wade Hindes}
\email{wmh33@txstate.edu}
\address{Department of Mathematics, Texas State University, San Marcos, TX 78666 USA}
\date{August 2025}
\begin{document}
\begin{abstract} Let $K$ be a function field of a curve in characteristic zero or a number field over which the $abc$-conjecture holds, fix $\alpha\in K$, and let $f_{d,c}(x)=x^d+c$ for some $d\geq2$ and some $c\in K$. Then for many $c$ and $d$, we prove that $f_{d,c}^n(x)-\alpha$ has at most $d$ factors in $K[x]$ for all $n\geq1$. For example, when $\alpha=0$ we prove that the set  
\[\Big\{d\,:\, f_{d,c}^n(x)\;\text{has at most $d$ factors in $K[x]$ for all $n\geq1$ and all $h(c)>0$}\Big\}\]
has positive asymptotic density. We then apply this result to compute the density of prime divisors in certain forward orbits and to establish the finiteness of integral points in certain backward orbits. 
\end{abstract} 
\maketitle
\section{Introduction} 
Let $K$ be a field, let $f\in K(x)$ be a rational function defined over $K$, and let $\alpha\in\mathbb{P}^1(K)$ be a $K$-rational point. It is a central problem in arithmetic dynamics to bound the number of Galois orbits of $f^{-n}(\alpha)$, the set of $n$th iterated preimages of $\alpha$ under $f$, as $n$ grows; see, for instance, \cite[\S 19]{benedetto2019current}. There are many known applications of this problem, including to the computation of arboreal Galois representations \cite{MR4188198,cubic:abc,Ferraguti:quad,Riccati,MR3220023}, to the calculation of densities of prime divisors in forward dynamical orbits \cite{MR3335237,Jones}, and to the finiteness of integral points in backward orbits \cite{jones2017eventually,sookdeo2011integer}. We study this problem when $f(x)=x^d+c$ and $K$ is a function field of a curve in characteristic zero or a number field over which the $abc$-conjecture \cite{elkies:1991} holds.          

In what follows, $\varphi(\cdot)$ and $\tau(\cdot)$ denote the Euler totient function and divisor counting function on $\mathbb{Z}$ respectively and $h(\cdot)$ denotes the Weil height function on $K$. 

\begin{thm}\label{thm:eventualstability+numberfields+abc}
Let $K$ be a number field over which the $abc$-conjecture holds, let $\alpha\in K$, let $d\geq2$, and let $f(x)=x^d+c$ for some $c\in K$. Then there exist constants $0<C_1(\alpha,K)<1$ and $C_2(\alpha,K)$ such that if the following conditions are all satisfied: \vspace{.15cm}  
\begin{enumerate}
\item[\textup{(1)}] $\varphi(d)>C_1(\alpha,K)d$, \vspace{.2cm}   
\item[\textup{(2)}] all of the prime factors $p|d$ satisfy $p>C_2(\alpha,K)$, \vspace{.2cm} 
\item[\textup{(3)}] $\alpha$ is not a fixed point of $f$, \vspace{.2cm}
\item[\textup{(4)}] $\min\big\{h(c),h(c-\alpha)\big\}>0$, \vspace{.2cm} 
\end{enumerate}
then $f^n(x)-\alpha$ has at most $\tau(d)$ factors in $K[x]$ for all $n\geq1$.  
\end{thm}
\begin{rem}\label{rem:non-prime+powers} We note that conditions (1) and (2) hold for a set of $d$'s of positive density, including all prime powers $d=p^m$ when $p$ is sufficiently large; see Corollary \ref{cor:evetual+stability+for+many+pairs}. More generally, if $t\geq1$ and $d$ is an integer with at most $t$ distinct prime factors, then there exists a constant $C_3(\alpha,K,t)$ such that both (1) and (2) hold whenever the prime factors $p|d$ satisfy $p>C_3(\alpha,K,t)$; see Lemma \ref{lem:primefactors}. Likewise, since we view $\alpha$ as fixed, condition (4) holds for all but finitely many $c\in K$. In particular, Theorem \ref{thm:eventualstability+numberfields+abc} is (to our knowledge at least) the first eventual stability result (even conditional result) for a large class of unicritical polynomials of non prime powered degree with nonzero basepoint; see \cite{MR4188198,bridy2021question,MR3335237,jones2017eventually} for prior results.\end{rem}
\begin{rem} It is known that the number of factors of $f^n(x)-\alpha$ is unbounded (i.e., the pair $(f,\alpha)$ is not eventually stable) when $\alpha$ is a periodic point for $f$; see page 3 of \cite{jones2017eventually}. However, it follows from the proof of Theorem \ref{thm:eventualstability+numberfields+abc} that conditions (2), (4) and the abc-conjecture imply that any periodic point is fixed, a property precluded by condition (3).        
\end{rem}
\begin{rem} In fact, the proof of Theorem \ref{thm:eventualstability+numberfields+abc} gives a complete description of the factorization of $f^n(x)-\alpha$ into irreducibles. For example, if  
\[m:=\max\Big\{m\,: m|d\;\;\text{and}\;\; c-\alpha=-y^m\;\text{for some}\; y\in K\Big\}, \]
then we prove that $f^n(x)-\alpha$ has exactly $\tau(m)$ factors. 
\end{rem}
\begin{rem} Even when $K=\mathbb{Q}$, $d=p$ is prime, and $\alpha=0$, more complicated factorization patterns occur in small degree. For example, if $f(x)=x^2-16/9$, then $f^n(x)$ has exactly $4$ irreducible factors for all $n\geq3$. On the other hand, Theorem \ref{thm:eventualstability+numberfields+abc} and the $abc$-conjecture suggest that this behavior \emph{can only occur in small degree}. Furthermore, in large degree the bound in Theorem \ref{thm:eventualstability+numberfields+abc} is sharp in general. For example, it is known that if $f(x)=x^p-y^p$ for some nonzero $y\in\mathbb{Z}$ and some prime $p$, then $f^n$ has exactly $\tau(p)=2$ factors over the rational numbers; see \cite[Corollary 5.2]{MR3335237}.     
\end{rem}
Moreover, we prove a stronger, unconditional result when $K$ is a function field of a curve: 
\begin{thm}\label{thm:functionFields} Let $K=k(C)$ be the function field of a curve where $k$ is a field of characteristic zero, let $d\geq2$, and let $f(x)=x^d+c$ for some $c\in K$. Then there exists a constant $C(\alpha,g)$ depending only on the genus $g$ of $K$ and $\alpha$ such that if the following conditions are satisfied:\vspace{.2cm} 
\begin{enumerate}
 \item[\textup{(1)}] all of the prime factors $p|d$ satisfy $p>C(\alpha,K)$, \vspace{.2cm} 
\item[\textup{(2)}] $\alpha$ is not a fixed point of $f$, \vspace{.2cm}
\item[\textup{(3)}] $\min\big\{h(c),h(c-\alpha)\big\}>0$, \vspace{.2cm} 
\end{enumerate}
then $f^n(x)-\alpha$ has at most $d$ factors in $K[x]$ for all $n\geq1$.
\end{thm}
In particular, we note that the set of degrees $d$ for which Theorems \ref{thm:eventualstability+numberfields+abc} and \ref{thm:functionFields} hold has positive lower density (i.e., the limit inferior of the proportion of positive integers $d\leq x$ satisfying conditions (1) and (2) of Theorem \ref{thm:eventualstability+numberfields+abc} as $x\rightarrow\infty$ is positive). Thus, we have established the eventual stability \cite{jones2017eventually} of many pairs of unicritical polynomials and points. 
\begin{cor}\label{cor:evetual+stability+for+many+pairs} 
Let $K$ be a function field of a curve in characteristic zero or a number field over which the $abc$-conjecture holds, let $\alpha\in K$, and let $f_{d,c}(x)=x^d+c$ for some $d\geq2$ and some $c\in K$. Moreover, let $K_\alpha$ be the set of $c\in K$ such that either $\min\{h(c),h(c-\alpha)\}=0$ or $c=\alpha-\alpha^m$ for some $m\geq2$ and let \vspace{.1cm}  
\[\mathcal{G}_\alpha:=\Big\{d\,:\, f_{d,c}^n(x)-\alpha\;\text{has at most $d$ factors in $K[x]$ for all $n\geq1$ and all $c\in K\setminus K_\alpha$}\Big\}.\vspace{.1cm}\]
Then $\mathcal{G}_\alpha$ has positive lower density for all $\alpha\in K$. In particular, when $\alpha=0$, we have that \vspace{.1cm} 
\[\Big\{d\,:\, f_{d,c}^n(x)\;\text{has at most $d$ factors in $K[x]$ for all $n\geq1$ and all $h(c)>0$}\Big\} \vspace{.1cm} \]
has positive lower density. 
\end{cor}
In fact, in the special case when $\alpha=0$ and $d=p$ is prime, sufficient control over the factorization of iterates implies certain density results for the prime factors in forward orbits; see, for example, \cite[Theorem 1.1]{MR3335237}. To make this point precise, we fix some notation. 

For a number field $K$, let $\mathcal{O}_K$ denote the ring of integers of $K$. Moreover, given a prime ideal $\mathfrak{p}\subset\mathcal{O}_K$, we let $v_\mathfrak{p}$ and $N(\mathfrak{p})=|\mathcal{O}_K/\mathfrak{p}\mathcal{O}_K|$ denote the corresponding valuation and norm of $\mathfrak{p}$ respectively. Likewise, to a polynomial $f\in K[x]$ and a point $\beta\in K$ we consider the set
\vspace{.1cm} 
\[
\mathcal{P}(f,K,\beta)=\big\{\mathfrak{p}\subset\mathcal{O}_K\,:\, v_{\mathfrak{p}}(f^n(b))>0\;\text{for some $n\geq1$ such that $f^n(\beta)\neq0$} \big\}
\vspace{.1cm} 
\]
of prime divisors in $K$ of the orbit of $\beta$ under $f$. In particular, it is expected that $\mathcal{P}(f,K,\beta)$ is a sparse set of primes in many (if not most) situations. Moreover, this property has applications to the dynamical Mordell-Lang conjecture \cite{benedetto2012case} and to questions about p-adic Mandelbrot set \cite{jones2007iterated}. Specifically, we wish to compute the (upper) natural density of the set of primes $\mathcal{P}(f,K,\beta)$,
\vspace{.1cm} 
\[\mathcal{D}(f,K,\beta)=\limsup_{X\rightarrow\infty}\frac{\#\{\mathfrak{p}\in\mathcal{P}(f,K,\beta)\,:\, N(\mathfrak{p})\leq X\}}{\#\{\mathfrak{p}\,:\, N(\mathfrak{p})\leq X\}}. \vspace{.1cm} 
\]
 In particular, if we view $f(x)=x^p+c$ for $c\in K$ with $h(c)>0$ and $p$ a prime as a polynomial over the extension $K(\zeta_p)/K$ where $\zeta_p$ is a primitive pth root of unity, then we may combine the proof of Theorem \ref{thm:eventualstability+numberfields+abc} with \cite[Theorem 1.1]{MR3335237} to prove that $\mathcal{P}(f,K(\zeta_p),\beta)$ has density zero within the full set of primes in $K(\zeta_p)$ for all $\beta\in K(\zeta_p)$ and all sufficiently large $p$ (depending only on the base field $K$).     
\begin{cor}\label{cor:density+roots+unity} Let $K$ be a number field over which the $abc$-conjecture holds, let $f(x)=x^p+c$ for some prime $p\geq2$ and $c\in K$ with $h(c)>0$, and let $\zeta_p\in\overline{K}$ be a primitive pth root of unity. Then $\mathcal{D}(f,K(\zeta_p),\beta)=0$ for all $p\gg_K0$ and all $\beta\in K(\zeta_p)$.   
\end{cor}
As a consequence, when $K=\mathbb{Q}$ we obtain the following partial result on the size of the prime divisors of orbits of $x^p+c$ without extending the base field to include roots of unity.\vspace{.1cm}  
\begin{cor}\label{cor:density+Q} Let $b,c\in\mathbb{Q}$ and let $f(x)=x^p+c$ for some prime $p$. Moreover, assume that the abc-conjecture holds over $\mathbb{Q}$. Then for all sufficiently large $p$, the set of primes $q\equiv1\Mod{p}$ that belong to $\mathcal{P}(f,\mathbb{Q},b)$ has density zero within the full set of primes $q\equiv1\Mod{p}$.  
\end{cor}
Likewise, we obtain the following (perhaps surprising) prediction that the orbit of zero under $x^p+c$ has positive (upper) density. \vspace{.1cm}   
\begin{cor}\label{cor:density+orbit+0} Let $c\in\mathbb{Q}^\times$ and let $f(x)=x^p+c$ for some prime $p$. Moreover, assume that the abc-conjecture holds over $\mathbb{Q}$. Then $\mathcal{D}(f,\mathbb{Q},0)=(p-2)/(p-1)$ for all sufficiently large $p$.   
\end{cor}
Finally, we may apply Theorem \ref{thm:eventualstability+numberfields+abc} to deduce the finiteness of integral points in certain backwards orbits. Namely, for $\alpha\in K$ and $f\in K[x]$, we define $O_f^{-}(\alpha)$ to be the set of all roots in $\overline{K}$ of the polynomials $f^n(x)-\alpha$ for all $n\geq1$. Moreover, given $\gamma\in K$ and a finite set $S$ of places of $K$ containing the archimedean ones, we say that $\beta\in\overline{K}$ is \emph{$S$-integral with respect to $\gamma$} if there is no prime $\mathfrak{q}$ of $K(\beta)$ lying over a prime outside of $S$ such that the images of $\gamma$ and $\beta$ modulo $\mathfrak{q}$ coincide. Likewise, we let $\mathcal{O}_{S,\gamma}$ be the set of $\beta\in\overline{K}$ that are $S$-integral with respect to $\gamma$.

In particular, we prove that the $abc$-conjecture implies that $O_f^{-}(\alpha)\cap\mathcal{O}_{S,\gamma}$ is finite for many $f(x)=x^d+c$ and many $\alpha,\gamma\in K$; see \cite{sookdeo2011integer} for a general conjecture along these lines as well as a proof in the case of power maps (i.e., $c=0$).  \vspace{.1cm}  
\begin{cor}\label{cor:integers+backwards+orbits} Let $K$ be a number field over which the $abc$-conjecture holds, let $\alpha\in K$, and let $f(x)=x^d+c$ for some $c\in K$ and $d\geq2$ satisfying conditions (1)-(4) of Theorem \ref{thm:eventualstability+numberfields+abc}. Moreover, assume that $\gamma\in K$ is not preperiodic for $f$ and that $S$ is a finite set of places of $K$ containing the archimedean ones. Then $O_f^{-}(\alpha)\cap\mathcal{O}_{S,\gamma}$ is finite.
\end{cor}
\begin{rem} In fact, increasing the constants in Theorem \ref{thm:eventualstability+numberfields+abc} if need be and assuming that $h(c)\gg_0 K$, we may conclude that $O_f^{-}(\alpha)\cap\mathcal{O}_{S,\gamma}$ is finite for all $\gamma\in K$ whenever $f^2(\gamma)\neq f(\gamma)$; this follows by combining Corollary \ref{cor:integers+backwards+orbits} with part (2) of \cite[Theorem 1.2]{PreperiodicPointsandABC}.  
\end{rem}
\vspace{.1cm} 
\noindent\textbf{Acknowledgements:} We thank Paul Pollack for pointing out the proof of Lemma \ref{lem:good+d's+positive+density}.  
\section{Proofs}
As in \cite{PreperiodicPointsandABC}, the main way we use the abc-conjecture \cite{elkies:1991} is to give effective height bounds for rational solutions to Fermat-Catalan equations; see \cite[Propsition 2.3]{PreperiodicPointsandABC} for stronger version of the following result.   
\begin{prop}\label{prop:FermatCatalan} Let $K$ be an $abc$-field, let $a,b,c\in K^*$, and assume that $ax^n+by^m=c$ for some $x,y,\in K$ and $m,n\geq2$. If $\max\{m,n\}\geq5$, then there exists constants $B_1(K)$ and $B_2(K)$ depending only on $K$ such that 
\[\max\{nh(x),mh(y)\}\leq B_1(K)\max\{h(a),h(b),h(c)\}+B_2(K).\] 
Moreover, when $K$ is a function field of a curve the constants $B_1$ and $B_2$ depend only on the genus of $K$.    
\end{prop}
\begin{rem} In fact, the assumption that $\min\{n,m\}\geq5$ can be weakened; however, the version above is sufficient for our purposes.     
\end{rem}
Likewise, we need the following irreducibility test for polynomials with a unicritical compositional factor; see \cite[Theorem 5.1]{PreperiodicPointsandABC} for a proof.  
\begin{prop}\label{prop:irreducbility} Let $K$ be a field of characteristic zero, let $g\in K[x]$ be monic and irreducible, and let $f=x^d+c$ for some $c\in K$ and some $d>1$. Moreover, assume that $g\circ f$ is reducible over $K$. Then one of the following statements must hold: \vspace{.1cm}
\begin{enumerate} 
\item If $d$ is not divisible by $4$, then there exists a prime $p|d$ such that $g(f(0))=\pm z^p$ for some $z\in K$. \vspace{.1cm} 
\item If $d$ is divisible by $4$, then $g(f(0))=(-1)^{e_1}4^{e_2}z^m$ for some $e_1,e_2\in\{0,1\}$, some divisor $m|d$ with $m>1$, and some $z\in K$. 
\end{enumerate}
\end{prop}
To prove Theorem \ref{thm:eventualstability+numberfields+abc} we need a few more basic facts, including the following lower bound on the height of a point in the orbit of zero.   
\begin{lem}\label{lem:orbit0} Let $f(x)=x^d+c$ for some $c\in K$ with $h(c)>0$. Then $h(f^m(0))\geq h(c)$ for all $m\geq1$ whenever $d$ is sufficiently large depending on $K$.     
\end{lem}
\begin{proof} It follows from standard properties of heights that
\begin{equation*}
h(f(\alpha))=h(\alpha^d+c)\geq dh(\alpha)-h(c)-\log2
\end{equation*}
holds for all $\alpha\in \overline{K}$. Moreover, we note that 
\begin{equation*}
(d-1)h(c)-\log(2)\geq h(c)
\end{equation*}
holds for all $d\gg_K0$. Indeed, the above bound holds for all $d> \log(2)/B_3(K)+2$ where $B_3(K)$ is the minimum positive height on $K$ (which is guaranteed to exist by Northcott). In particular, we claim that $h(f^m(0))\geq h(c)$ for all $m\geq1$ and all $d$ sufficiently large. To see this, we proceed by induction; note that when $m=1$ we obtain $h(c)\geq h(c)$. On the other hand, assuming $h(f^m(0))\geq h(c)$ and $d>\log(2)/B_3(K)+2$ we see that 
\[h(f^{m+1}(0))=h(f^m(0)^d+c)\geq dh(f^m(0))-h(c)-\log(2)\geq (d-1)h(c)-\log(2)\geq h(c)\]
as desired.    
\end{proof}
In particular, we can immediately prove that in large degree, an irreducible polynomial $f=x^d+c$ is stable (meaning all of iterates are irreducible); compare to \cite[Proposition 5.3]{PreperiodicPointsandABC}.  
\begin{prop}\label{prop:stability} Let $K$ be an $abc$-field, let $\alpha\in K$, and let $f=x^d+c$ for some $c\in K$ with $h(c)>0$. Then for all $d\gg_{K,\alpha}0$, the polynomial $f(x)-\alpha$ is irreducible over $K$ if and only if $f^n(x)-\alpha$ is irreducible over $K$ for all $n\geq1$.   
\end{prop}
\begin{proof} Assume that $h(c)>0$, that $d\geq5$ is sufficiently large so that Lemma \ref{lem:orbit0} holds, that $f(x)-\alpha$ is irreducible over $K$, and that $f^n(x)-\alpha$ is reducible over $K$ for some $n\geq2$. Moreover, we may assume that $n$ is the minimum such iterate. Then  Proposition \ref{prop:irreducbility} applied to the pair $g=f^{n-1}-\alpha$ and $f$ implies that $f^n(0)-\alpha=rz^m$ for some some $z\in K$, some $m\geq2$, and some $r\in K^*$ with $h(r)\leq\log(4)$. Let $X=f^{n-1}(0)$, so that $X^d+c-\alpha=rz^m$. Note that $c-\alpha\neq0$, since otherwise $f(x)-\alpha=x^d+c-\alpha=x^d$ is reducible, a contradiction. Hence, Proposition \ref{prop:FermatCatalan} and Lemma \ref{lem:orbit0} together imply that
\begin{equation*}
\begin{split}
dh(c)&\leq dh(X)\leq B_1(K)\max\{h(c-\alpha),\log(4)\}+B_2(K)\\[5pt]
&\leq B_1(K)(h(c)+h(\alpha)+\log4)+B_2(K)
\end{split} 
\end{equation*}
But then $d\leq B_1(K)(1+h(\alpha)/B_3(K)+(\log4)/B_3(K))+B_2/B_3(K)=B'(\alpha,K)$, and $d$ is bounded by a constant depending only on $\alpha$ and $K$ as claimed; here $B_3(K)$ once again denotes the minimum positive height of an element of $K$.    
\end{proof}
Next, we use the following result, which ensures that the absolute Galois group of a number field $K$ acts transitively on the primitive roots of unity of a given order.  
\begin{lem}\label{lem:full+Galois+group} Let $K$ be a number field with discriminant $\delta_K$ and let $\mu_m$ be a complete set of $m$th roots of unity in $\overline{K}$ for some $m>1$. If $\gcd(\delta_K,m)=1$, then $\Gal(K(\mu_m)/K)\cong(\mathbb{Z}/m\mathbb{Z})^*$.      
\end{lem}
\begin{proof} The case of $K=\mathbb{Q}$ follows from the irreducibility of the $m$th cyclotomic polynomial. Hence, may assume that $[K:\mathbb{Q}]\geq2$ and that $\gcd(\delta_K,m)=1$. Now let $L:=\mathbb{Q}(\mu_m)\cap K$ and assume that $p$ is a rational prime in $\mathbb{Q}$ that ramifies in $L$. Then $p$ must ramify in both $K$ and $\mathbb{Q}(\mu_m)$. But then $p|\delta_K$ and $p|m$, a contradiction. Hence, $L/\mathbb{Q}$ is an unramified extension so that $L=\mathbb{Q}$. It now follows from the second isomorphism theorem in Galois theory that 
\[\Gal(K(\zeta_m)/K)=\Gal(\mathbb{Q}(\zeta_m) K/K)\cong\Gal(\mathbb{Q}(\zeta_m)/L)=\Gal(\mathbb{Q}(\zeta_m)/\mathbb{Q})\cong(\mathbb{Z}/m\mathbb{Z})^*\]
as desired. 
\end{proof}
Lastly, we prove that the set of possible degrees $d$ for which Theorem \ref{thm:eventualstability+numberfields+abc} may be applied has positive lower asymptotic density. 
\begin{lem}\label{lem:good+d's+positive+density} Let $0<C_1<1$ and $C_2$ be constants. Then the set of positive integers $d$ satisfying both of the following conditions:
\[(1)\;\;\;\varphi(d)>C_1d\;\;\;\;\;\text{and}\;\;\;\;\;\;(2)\;\;\; p>C_2\;\;\text{for all primes $p|d$},\]
has positive lower asymptotic density. 
\end{lem}
\begin{proof} We study the function $\log(d/\varphi(d))$ rather than $\varphi(d)/d$. Indeed, if $d$ is such that $\varphi(d)/d$ is close to $1$, then $\log(d/\phi(d))$ is close to $0$ and vice versa. Moreover, the function $\log(d/\varphi(d))$ is additive, and for each prime power $d=p^k$, we have that
\[\log\Big(\frac{p^k}{\varphi(p^k)}\Big) = -\log\Big(1-\frac{1}{p}\Big) <\frac{2}{p}.\]
In particular, we compute that  
\begin{equation}\label{eq:upper}
\log\Big(\frac{d}{\varphi(d)}\Big) < 2 \sum_{p | d} \frac{1}{p}.
\end{equation}
So it suffices to produce many $d$ with only large prime factors such that $\sum_{p|d} 1/p$ is small. Moreover, since $\log(d/\phi(d))\geq0$, this can be achieved by a first moment counting argument. 

Fix a large positive integer $M$ and let $S$ be the set of all positive integers whose prime factors are all greater than $M$. Then it is well known that $S$ has positive density given by
\[c_M = \prod_{p\leq M} \Big(1-\frac{1}{p}\Big).\]
On the other hand, for all $x>0$ we see that \vspace{.15cm} 
\begin{equation*}
\sum_{\substack{d\in S\\ d \leq x}} \sum_{p | d} \frac{1}{p}\leq\sum_{d\leq x} \sum_{\substack{p | d\\ p > M}} \frac{1}{p} \leq \sum_{M < p\leq x} \frac{1}{p} \sum_{\substack{d\leq x\\ p | d}} 1\leq x \sum_{M < p\leq x} \frac{1}{p^2} <x \sum_{m > M} \frac{1}{m(m-1)} = \frac{x}{M}.\vspace{.15cm} 
\end{equation*}
Hence, we may conclude that  
\[\#\bigg\{d\in S\;:\; d\leq x\;\;\;\text{and}\;\;\sum_{p|d} \frac{1}{p} > \frac{1}{\sqrt{M}}\bigg\}\leq\frac{x}{\sqrt{M}}\]
Next we compare $c_M$ to $1/\sqrt{M}$. Mertens' third theorem states that $c_M$ approaches $e^{-\gamma}/\log M$ as $M\rightarrow\infty$, where $\gamma$ is the Euler–Mascheroni constant. In particular, $c_M>2/\sqrt{M}$ holds for all $M$ sufficiently large. We conclude that if $M$ is fixed and sufficiently large, then for all large $x$, there are at least $(1/2)c_Mx$ numbers $d\leq x$ with $d\in S$ and $\sum_{p | d} 1/p\leq 1/\sqrt{M}$. Moreover, for the $d$ satisfying this last inequality, we have from \eqref{eq:upper} that  
\[\varphi(d)/d > \exp(-2/\sqrt{M}).\]
In particular, by choosing any $M>C_2$ sufficiently large so that $\exp(-2/\sqrt{M})>C_1$, we see that the density of those $d$ satisfying both (1) and (2) in Lemma \ref{lem:good+d's+positive+density} is positive as claimed.  
\end{proof}
We now have all of the tools in place to prove our main result. 
\begin{proof}[(Proof of Theorem \ref{thm:eventualstability+numberfields+abc})] Let $K$ be a number field over which the $abc$-conjecture holds, and let $\delta_K$ and $\mu_K$ denote the discriminant and the set of roots of unity of $K$ respectively. Moreover, fix $\alpha\in K$ and let $f(x)=x^d+c$ for some $c\in K$. Furthermore, we assume that $\alpha$ is not a fixed point of $f$, that $d\gg_K0$ so that Lemma \ref{lem:orbit0} holds, that $\gcd\big(d,6\cdot\delta_K\cdot|\mu_K|\big)=1$, and that $\min\{h(c),h(c-\alpha)\}>0$. Now define 
\[m:=\max\big\{m\,: m|d\;\;\text{and}\;\; c-\alpha=-y^m\;\text{for some}\; y\in K\big\}.\]
Note that if $m=1$, then $f-\alpha$ is irreducible over $K$ by \cite[Theorem 9.1]{MR1878556}. In this case, $f^n-\alpha$ is irreducible over $K$ for all $n\geq1$ and all $d\gg_{K,\alpha}0$ by Proposition \ref{prop:stability}, without further assumption on $\varphi(d)$. In particular, we may assume that $m>1$. Then we obtain the factorization  
\[f(x)-\alpha=x^d+c-\alpha=x^d-y^m=\prod_{\zeta\in\mu_{m}}(x^r-\zeta y)\]
of $f$ over $K(\mu_m)$, where $r=d/m$ and $\mu_{m}$ denotes the set of all $m$th roots of unity in $\overline{K}$. Note next that each polynomial $g_\zeta(x):=x^{r}-\zeta y$ must be irreducible over $K(\zeta)$. If not, then $r>1$ and so \cite[Theorem 9.1]{MR1878556} implies that $\zeta y=z^p$ for some $z\in K(\zeta)$ and some prime $p|r$. But since $\zeta\in\mu_m$ is also a $p$th power in $K(\mu_{mp})$, we see that $y$ must be a $p$th power in $K(\mu_{mp})$. Hence, $K\subseteq K(\sqrt[p]{y})\subseteq K(\mu_{mp})$ and so $K(\sqrt[p]{y})/K$ must be an Abelian extension of $K$. Thus, \cite[Theorem 2]{schinzel1977abelian} implies that $y=u^p$ for some $u\in K$; here we use that $d$ (and so also $p$) is coprime to $|\mu_K|$. But then $c-\alpha=-y^m=-u^{pm}$, which contradicts the definition of $m$; note that $pm$ divides $d$ since $pt=r=d/m$ for some $t\in\mathbb{Z}$. Therefore, the polynomials $g_\zeta$ are irreducible over $K(\zeta)$ for all $\zeta\in\mu_m$ as claimed. 

Similarly, we wish to show that $g_\zeta\circ f^n$ is irreducible over $K(\zeta)$ for all $\zeta\in\mu_m$ and all $n\geq1$. If not, then repeated application of Proposition \ref{prop:irreducbility} implies that 
\begin{equation}\label{eq:pthpower}
f^n(0)^r-\zeta y=g_\zeta\circ f^n(0)=z^p
\end{equation} 
for some $n\geq1$, some $\zeta\in\mu_m$, some $p|d$, and some $z\in K(\zeta)$. Moreover, we note that $p\geq5$ by our assumption that $\gcd(d,6)=1$. From here, we proceed in cases to show that either condition (1) or (2) of Theorem \ref{thm:eventualstability+numberfields+abc} must fail. In what follows, we recall that $r=d/m$. \\[5pt]
\textbf{Case (1):} Suppose that $\zeta=1$ and that $r=1$. Now when $n=1$, we see from \eqref{eq:pthpower} that 
\[-y^d-(y-\alpha)=-y^m-(y-\alpha)=(c-\alpha)-(y-\alpha)=c-y=f(0)-y=g_\zeta\circ f(0)=z^p.\]
On the other hand, if $y-\alpha=0$, then $c-\alpha=-y^m=-y^d=-\alpha^d$ and so $\alpha$ is a fixed point of $f$, a contradiction. In particular, Proposition \ref{prop:FermatCatalan} applied to the equation $-(y-\alpha)=z^p+y^d$ implies that $dh(y)\leq B_1(K)h(y-\alpha)+B_2(K)$ for some $B_1(K)$ and $B_2(K)$ depending on $K$. Moreover, we note that $h(y)>0$, since otherwise $h(c-\alpha)=h(-y^d)=dh(y)=0$, and we again obtain a contradiction. Then Northcott's theorem implies that $h(y)\geq B_3(K)>0$ for some constant $B_3(K)$ depending only on the degree of $K$. Hence, \[d\leq B_1(K)(1+h(\alpha)/B_3(K)+\log2/B_3(K))+B_2(K)/B_3(K),\] 
so that $d$ is bounded by a constant depending only on $K$ and $\alpha$. Likewise, when $n\geq2$ then \eqref{eq:pthpower} implies that 
\[f^{n-1}(0)^d+c-y=f(f^{n-1}(0))-y=f^n(0)-y=g_\zeta\circ f^n(0)=z^p.\] 
Therefore, if $c-y\neq 0$, then Proposition \ref{prop:FermatCatalan} and Lemma \ref{lem:orbit0} together imply that \vspace{.1cm}  
\begin{equation*}
\begin{split}
d h(c)\leq dh(f^{n-1}(0))&\leq B_1(K)h(c-y)+B_2(K)\\[5pt] 
&\leq B_1(K)\Big(h(c)+h(y)+\log2\Big)+B_2(K)\\[5pt] 
&\leq B_1(K)\Big(h(c)+dh(y)+\log2\Big)+B_2(K)\\[5pt] 
&\leq B_1(K)\Big(h(c)+h(c-\alpha)+\log2\Big)+B_2(K)\\[5pt]
&\leq B_1(K)\Big(2 h(c)+h(\alpha)+\log4\Big)+B_2(K)
\end{split} 
\end{equation*}
In particular, $d\leq B_1(K)(2+h(\alpha)/B_3(K)+\log4/B_3(K))+B_2(K)/B_3(K)$, so that $d$ is bounded by a constant depending $K$ and $\alpha$. Finally, if $c-y=0$, then $c-\alpha=-y^d=-c^d$. Thus $\alpha=c^d+c$ so that \[h(\alpha)=h(c^d+c)\geq h(c^d)-h(c)-\log2= (d-1)h(c)-\log2> (d-1)B_3(K)-\log2,\] 
and $d$ is once again bounded by a constant depending on $K$ and $\alpha$. In particular, if we assume that every prime factor of $d$ is sufficiently large depending on $K$ and $\alpha$ (forcing $d$ to be large also), then it is not possible for \eqref{eq:pthpower} to hold in this case. \\[5pt]  
\textbf{Case (2):} Suppose that $\zeta= 1$ and $r>1$. Then, if we set $X=f^n(0)$, we have $X^r-y=z^p$ by \eqref{eq:pthpower}. Moreover, $h(y)>0$, since otherwise $h(c-\alpha)=h(-y^m)=0$, a contradicition. Thus, Proposition \ref{prop:FermatCatalan} and Lemma \ref{lem:orbit0} together imply that 
\begin{equation*}
\begin{split} 
r h(c)\leq rh(X)&\leq B_1(K)h(y)+B_2(K) \\[5pt] 
&\leq B_1(K)h(-y^m)+B_2(K) \\[5pt] 
&\leq B_1(K)(h(c)+h(\alpha)+\log2)+B_2(K).
\end{split}
\end{equation*}
Hence, $r\leq B_1(K)(1+h(\alpha)/B_3(K)+\log2/B_3(K))+B_2(K)/B_3(K)$ is bounded by a constant depending on $K$ and $\alpha$. In particular, if we assume that every prime factor of $d$ is sufficiently large depending on $K$ and $\alpha$ (forcing $r$ to be large also), then it is not possible for \eqref{eq:pthpower} to hold in this case as well.       
\\[5pt] 
\textbf{Case (3):} Suppose that $\zeta\neq1$. Then $\zeta$ is a primitive $a$th root of unity for some divisor $a|m$ with $a>1$. Hence, $\Gal(K(\zeta)/K)=\Gal(K(\mu_a)/K)\cong (\mathbb{Z}/a\mathbb{Z})^*$ by Lemma \ref{lem:full+Galois+group} and our assumption that $d$ (and so $a$ also) is coprime to $\delta_K$. Now, let $X=f^n(0)$, so that $X^r-\zeta y=z^p$. Then applying any $\sigma\in\Gal(K(\zeta)/K)$ to the equation $X^r-\zeta y=z^p$ we see that $X^r-\omega y$ is also a $p$th power in $K(\zeta)$ for any primitive $a$th root of unity $\omega$ (since $\Gal(K(\zeta)/K)$ acts transitively on the primitive $a$th roots of unity by assumption). Hence, we deduce that 
\begin{equation*}
\begin{split} 
X^{ra}-y^a&=\prod_{\omega\in \mu_a}(X^r-\omega y)\\[3pt] 
&=\prod_{\substack{\omega\in\mu_a\\ o(\omega)<a}}X^r-\omega y\cdot \prod_{\substack{\omega\in\mu_a\\ o(\omega)=a}}X^r-\omega y\\[3pt] 
&= \bigg(\prod_{\substack{\omega\in\mu_a\\ o(\omega)<a}}X^r-\omega y\bigg)\cdot \bigg(\prod_{\sigma\in \Gal(K(\zeta)/K)}\hspace{-.5cm}\sigma(z)\bigg)^p\\[3pt] 
&=G(X,y)\cdot u^p 
\end{split} 
\end{equation*}
for some $u\in K$, where $o(\omega)$ denotes the multiplicative order of a root of unity $\omega$; here $G(X,y)\in K$ is the product of $X^r-\omega y$ over all $\omega\in\mu_a$ with $o(\omega)<a$. In particular, since $y\neq0$ as $h(c-\alpha)>0$ by assumption and $c-\alpha=-y^m$ by defintion of $y$ and $m$, it follows from Proposition \ref{prop:FermatCatalan} and Lemma \ref{lem:orbit0} that \vspace{.1cm} 
\begin{equation}\label{eq:numberfield1}
\begin{split}
rah(X)&\leq B_1(K)\max\{h(G(X,y)),ah(y)\}+B_2(K)\\[5pt] 
&\leq B_1(K)\max\{h(G(X,y)),m h(y)\}+B_2(K)\\[5pt] 
&= B_1(K)\max\{h(G(X,y)),h(c-\alpha)\}+B_2(K)\\[3pt] 
\end{split} 
\end{equation} 
On the other hand, $h(X^r-\omega y)\leq rh(X)+h(y)+\log(2)$ for all $\omega$, so that \vspace{.1cm}
\begin{equation}\label{eq:numberfield2}
\begin{split} 
h(G(X,y))&\leq (a-\varphi(a))\Big(rh(X)+h(y)+\log2\Big)\\[5pt]
& \leq (a-\varphi(a))\Big(rh(X)+\frac{1}{m}h(c-\alpha)+\log2\Big)\\[5pt]
&\leq (a-\varphi(a))\Big(rh(X)+h(c-\alpha)+\log2\Big)\\[3pt] 
\end{split} 
\end{equation} 
where $\varphi(a)=|(\mathbb{Z}/a\mathbb{Z})^*|$. Hence, combining \eqref{eq:numberfield1}, \eqref{eq:numberfield2}, and Lemma \ref{lem:orbit0}, we see that \vspace{.1cm} 
\begin{equation}\label{eq:numberfield3}
\begin{split} 
ra h(X)&\leq B_1(K) (a-\varphi(a))\Big(rh(X)+h(c-\alpha)+\log2\Big)+B_2(K)\\[5pt] 
&\leq B_1(K)(a-\varphi(a))\Big(rh(X)+h(c)+h(\alpha)+\log4\Big)+B_2(K)\\[5pt]
&\leq B_1(K)(a-\varphi(a))\Big(2rh(X)+\frac{h(\alpha)+\log4}{B_3(K)}rh(X)\Big)+B_2(K)\\[5pt]
&=B_1(K)(a-\varphi(a))\Big(2+\frac{h(\alpha)+\log4}{B_3(K)}\Big)rh(X)+B_2(K)\\[5pt]
&=B_4(K,\alpha)(a-\varphi(a))rh(X)+B_2(K)\\[3pt] 
\end{split} 
\end{equation}
for some constant $B_4(K,\alpha)\geq1$ depending on both $K$ and $\alpha$. Now suppose that $d$ satisfies both of the following conditions:\vspace{.1cm}  
\begin{equation}\label{eq:numberfield4}
\textup{(I)}\,\;\;\;\varphi(d)>\Bigg(\frac{2B_4(K,\alpha)-1}{2B_4(K,\alpha)}\Bigg)d\;\;\;\;\;\text{and}\;\;\;\;\; \textup{(II)}\,\;\;\; p>\frac{3B_2(K)}{B_3(K)}\; \text{for all primes $p|d$.}
\end{equation}
Then since $a|d$ and $a>1$, we have that (I) holds when $d$ is replaced with $a$ and (II) holds when $p$ is replaced with $a$. In particular, \eqref{eq:numberfield3} and \eqref{eq:numberfield4} together imply that
\vspace{.1cm} 
\[
1\leq B_4(K,\alpha)\bigg(\frac{a-\varphi(a)}{a}\bigg)+\frac{B_2(K)}{rah(X)}\leq \frac{1}{2}+\frac{B_2(K)}{a B_3(K)}\leq \frac{1}{2}+\frac{1}{3},
\vspace{.1cm} 
\]
and we obtain a contradiction.

To summarize: there are constants $C_1(K,\alpha)$ and $C_2(K,\alpha)$ depending only on $K$ and $\alpha$ such that if conditions (1)-(4) of Theorem \ref{thm:eventualstability+numberfields+abc} are all satisfied, then $g_\zeta\circ f^n$ is irreducible in $K(\zeta)[x]$ for all $\zeta\in\mu_m$ and all $n\geq1$; in fact, we showed  the stronger fact that \textbf{$g_\zeta\circ f^n(0)$ cannot be a $p$th power in $K(\zeta)$ for any prime $p|d$}. Finally, for any divisor $a|d$ we let 
\begin{equation}\label{eq:overK}
g_a(x):=\prod_{o(\zeta)=a}g_\zeta(x)=\prod_{o(\zeta)=a}(x^r-\zeta y)\in K[x].
\end{equation}
Here, $g_a$ has coefficients in $K$ since any $\sigma\in Gal(\overline{K}/K)$ preserves the order of a root of unity and $y\in K$. We claim that $g_a\circ f^n\in K[x]$ is irreducible for all $a|m$ and all $n\geq0$. To see this, fix $a>1$ and let $\beta_1,\beta_2\in\overline{K}$ be any two roots of $g_a\circ f^n$; the case when $a=1$ follows apriori. In particular, it follows from the factorization in \eqref{eq:overK} that there are $\zeta_1,\zeta_2$ with $o(\zeta_i)=a$ such that $\beta_1$ and $\beta_2$ are roots of $g_{\zeta_1}\circ f^n$ and $g_{\zeta_2}\circ f^n$ respectively. Now since $\Gal(K(\mu_a)/K)\cong(\mathbb{Z}/a\mathbb{Z})^*$ there exists $\sigma\in\Gal(\overline{K}/K)$ such that $\sigma(\zeta_1)=\zeta_2$. Moreover, 
\[0=\sigma(0)=\sigma(g_{\zeta_1}(f^n(\beta_1)))=g_{\sigma(\zeta_1)}(f^n(\sigma(\beta_1)))=g_{\zeta_1}(f^n(\sigma(\beta_1))).\]
Hence, $\sigma(\beta_1)$ is a root of $g_{\zeta_2}\circ f^n$, an irreducible polynomial over $K(\mu_a)$. Thus there exists $\rho\in\Gal(\overline{K}/K(\mu_a))\subseteq \Gal(\overline{K}/K)$ such that $\rho(\sigma(\beta_1))=\beta_2$. Therefore, $\Gal(\overline{K}/K)$ acts transitively on the roots of $g_a\circ f^n$, and so it is irreducible over $K$ as claimed. In particular, we deduce that 
\[f^{n}(x)-\alpha=\prod_{a|m}g_a\circ f^{n-1}(x)\]
is a complete factorization into irreducible polynomials in $K[x]$ for all $n\geq1$. Moreover, it follows that $f^n$ has at most $\tau(m)\leq\tau(d)$ factors over $K$ for all $n\geq1$. 
\end{proof}
Next, we justify Remark \ref{rem:non-prime+powers}, that Theorem \ref{thm:eventualstability+numberfields+abc} applies to many non-prime powered degrees. 
\begin{lem}\label{lem:primefactors} Let $t\geq1$, let $0<\epsilon<1$, and let $d$ be an integer with at most $t$ distinct prime factors. Then exists a constant $M(t,\epsilon)$ depending only on $t$ and $\epsilon$ such that if every prime divisor $p|d$ satisfies $p> M(t,\epsilon)$, then $\varphi(d)>\epsilon d$.     
\end{lem}
\begin{proof} Fix $t\geq1$ and $0<\epsilon<1$ and let $d=p_1^{e_1}\cdot p_t^{e_t}$ be a factorization of $d$ into primes. Moreover, assume that $p_1,\dots,p_t>M$ where $M:=t/(1-\epsilon)$. Then it follows from Bernoulli's inequality, $(1+x)^t\geq 1+tx$ for all $x\geq-1$, applied to $x=-1/M$ that   
\[\frac{\varphi(d)}{d}=\Big(1-\frac{1}{p_1}\Big)\cdots\Big(1-\frac{1}{p_t}\Big)\geq(1-1/M)^t\geq1-t\Big(\frac{1}{M}\Big)=\epsilon \]
as desired. 
\end{proof}
Before we apply Theorem \ref{thm:eventualstability+numberfields+abc} to compute the density of prime divisors for certain dynamical orbits, we need the following fact: 
\begin{rem}\label{rem:pth+power} Let $K$ be a field of characteristic zero, let $y\in K$, and let $p$ be an odd prime. Then we note that $y$ is a $p$th power in $K$ if and only if it is a $p$th power in $K(\mu_p)$. Suppose for a contradiction that $y$ is not a $p$th power in $K$ and that $y=z^p$ for some $z\in K(\mu_p)$. In particular, \cite[Theorem 9.1]{MR1878556} implies that $x^p-y$ is irreducible over $K$ so that $[K(z):K]=p$. On the other hand,  $K(z)\subseteq K(\zeta_p)$ so that $p=[K(z):K]\leq [K(\zeta_p):K]\leq p-1$, a contradiction.       
\end{rem}
\begin{proof}[(Proof of Corollary \ref{cor:density+roots+unity})]
Assume that the $abc$-conjecture holds over $K$ and that $p$ is a prime. Then when $f=x^p+c$ is irredcuible over $K$ and $h(c)>0$, it follows form the proof of Proposition \ref{prop:stability} that $f^n(0)$ is not a $p$th power in $K$ for all $p\gg_K0$. Hence, $f^n(0)$ is not a $p$th power in $K(\zeta_p)$ for all $n\geq1$ by Remark \ref{rem:pth+power}. In particular, it follows from \cite[Theorem 1.1]{MR3335237} that $\mathcal{P}\big(f,K(\zeta_p),\beta\big)$ has natural density zero for all $\beta\in K(\zeta_p)$ as claimed. Likewise when $f$ is reducible over $K$, then it follows from the proof of Theorem \ref{thm:eventualstability+numberfields+abc} that $f=g_1\cdots g_p$ for some linear $g_i\in K(\zeta_p)[x]$ such that $g_i(f^m(0))$ is not a $p$th power in $K(\zeta_p)$ for all $1\leq i\leq p$ and all $m\geq1$; here we use Remark \ref{rem:pth+power} in the case when $g(x)=x-y$ for $y\in K$ with $c=-y^p$. It therefore follows from \cite[Theorem 1.1]{MR3335237} that $\mathcal{P}\big(f,K(\zeta_p),\beta\big)$ has Dirichlet density zero for all $\beta\in K(\zeta_p)$ as claimed.    
\end{proof}
\begin{proof}[(Proof of Corollary \ref{cor:density+Q})] We argue as in the paragraph proceeding \cite[Corollary 1.3]{MR3335237}. First we recall that a rational prime $q$ splits completely in $\mathbb{Q}(\zeta_p)$ if and only if $q\equiv 1\Mod{p}$. Second, we note that the set of primes $\mathfrak{q}$ of $\mathbb{Q}(\zeta_p)$ that lie over a rational prime $q$ which splits completely in $\mathbb{Q}(\zeta_p)$ has full density, since such a prime necessarily has norm $q$ while the norm of any of other prime over $q'$ is at least $(q')^2$. In particular, if $T$ denotes the set of primes of $\mathbb{Q}(\zeta_p)$ that lie over a prime $q\equiv1\Mod{p}$ and $c\in\mathbb{Q}\setminus\{0,\pm{1}\}$ and $b\in\mathbb{Q}$, then it follows from Corollary \ref{cor:density+roots+unity} that \vspace{.1cm} 
\begin{equation*}
\begin{split}
0&=\limsup_{X\rightarrow\infty}\frac{\#\Big\{\mathfrak{q}\in T\,:\, v_{\mathfrak{q}}(f^n(b))>0\;\text{for some $n\geq0$}\; \text{and}\; N(\mathfrak{q})\leq X\Big\}}{\#\Big\{\mathfrak{q}\in T\,:\,N(\mathfrak{q})\leq X\Big\}}\\[5pt]
&=\limsup_{X\rightarrow\infty}\frac{(p-1)\cdot\#\Big\{q\equiv1\Mod{p}\,:\, v_q(f^n(b))>0\;\text{for some $n\geq0$}\; \text{and}\; q\leq X\Big\}}{(p-1)\cdot \#\Big\{q\equiv1\Mod{p}\,:\,q\leq X\Big\}}\\[5pt]
0&=\limsup_{X\rightarrow\infty}\frac{\#\Big\{q\equiv1\Mod{p}\,:\, v_q(f^n(b))>0\;\text{for some $n\geq0$}\; \text{and}\; q\leq X\Big\}}{\#\Big\{q\equiv1\Mod{p}\,:\,q\leq X\Big\}}\\[2pt] 
\end{split} 
\end{equation*}
holds for all $p$ sufficiently large as claimed. Similarly, since the excluded values of $c$ are all integral, one may apply \cite[Corollary 1.3]{MR3335237} directly in these cases to reach the same conclusion. 
\end{proof}
\begin{proof}[(Proof of Corollary \ref{cor:density+orbit+0})] We argue as in \cite[Corollary 1.4]{MR3335237} and sieve prime divisors of orbits by their congruence modulo $p$. Assume that $c\neq0$ and $c\neq-1$ if $p=2$ and let $q$ be a prime number not dividing the denominator of $c$. If $q\not\equiv 1\Mod{p}$, then $f(x)=x^p+c$ acts as a permutation on $\mathbb{F}_q$. Hence, every element $b\in\mathbb{F}_q$ has a \emph{periodic orbit} under $f$. In particular, this is true of $b=0$. Hence, there exists an $n_q\geq1$ such that $f^{n_q}(0)\equiv0\Mod{q}$. Moreover, we note that $f^{n_q}(0)\neq0$ since we have excluded $c=0$ and $c=-1$ when $p=2$, the only cases where $0$ is periodic for $x^p+c$ over $\mathbb{Q}$. In particular, we deduce that $q\in\mathcal{P}(f,\mathbb{Q},0)$. In short, we have shown that for all but finitely many $q\not\equiv1\Mod{p}$, it must be the case that $q\in\mathcal{P}(f,\mathbb{Q},0)$. In particular, it follows that $\mathcal{D}(f,\mathbb{Q},0)\geq (p-2)/(p-1)$. On the other hand, Corollary \ref{cor:density+Q} implies that $\mathcal{D}(f,\mathbb{Q},0)\leq (p-2)/(p-1)$ when $p$ is sufficiently large, and the claim follows.            
\end{proof}
\begin{proof}[(Proof of Corollary \ref{cor:integers+backwards+orbits})] Under the conditions of Theorem \ref{thm:eventualstability+numberfields+abc}, we have that the pair $(f,\alpha)$ is eventually stable, meaning that the number of factors of the polynomial $f^n(x)-\alpha$ is bounded independently of $n$. Hence \cite[Theorems 2.5 and 2.6]{sookdeo2011integer} imply that $O_f^{-}(\alpha)\cap\mathcal{O}_{S,\gamma}$ is finite for all non-preperiodic $\gamma\in K$; see also \cite[Theorem 3.1]{jones2017eventually}.    
\end{proof}
\section{Function Fields}
In what follows, $k$ is a field of characteristic zero, $t$ is a transcendental over $k$, and $K/k(t)$ is a finite extension; equivalently, $K$ is the function field of a curve over $k$. Moreover, we will assume that the constant field $k$ is algebraically closed, since any bound on the number of irreducible factors of $f^n(x)-\alpha$ in $K[x]$ in this case gives an upper bound on the number of irreducible factors in $K[x]$ when $k$ is not necessarily closed.
\begin{rem} The proof of Theorem \ref{thm:functionFields} follows that of Theorem \ref{thm:eventualstability+numberfields+abc} very closely, except that Case (3) in the proof of Theorem \ref{thm:eventualstability+numberfields+abc} is not needed; moreover, this is the main reason that a stronger statement is possible in the function field case (i.e., with fewer stipulations on $d$).      
\end{rem}
\begin{proof}[(Proof of Theorem \ref{thm:functionFields})] Let $\alpha\in K$ and let $f(x)=x^d+c$ for some $d\geq2$ and some $c\in K$. Moreover, assume that $\alpha$ is not a fixed point of $f$, that $d$ is not divisible by $2$ or $3$ and that $d\gg_K0$ so that Lemma \ref{lem:orbit0} holds, and that $\min\{h(c),h(c-\alpha)\}>0$; equivalently $c$ and $c-\alpha$ are nonconstant functions. Now define
\[m:=\max\Big\{m\,: m|d\;\;\text{and}\;\; c-\alpha=-y^m\;\text{for some}\; y\in K\Big\}.\]
Note that if $m=1$, then $f$ is irreducible over $K$ by \cite[Theorem 9.1]{MR1878556}. In this case, $f^n-\alpha$ is irreducible over $K$ for all $n\geq1$ and all $d\gg_{K,\alpha}0$ by Proposition \ref{prop:stability}. In particular, we may assume that $m>1$. Now, since $k$ contains a complete set of roots of unity, we obtain the factorization  
\[f(x)-\alpha=x^d+c-\alpha=x^d-y^m=\prod_{\zeta\in\mu_{m}}(x^r-\zeta y)\]
of $f$ in $K[x]$, where $r=d/m$ and $\mu_{m}$ denotes the $m$th roots of unity in $K$. Note next that each polynomial $g_\zeta(x):=x^{r}-\zeta y$ must be irreducible in $K[x]$ for all $\zeta\in\mu_{m}$. If not, then $r>1$ and so \cite[Theorem 9.1]{MR1878556} implies that $\zeta y=z^p$ for some prime $p|r$ and some $z\in K$. Then, since $k$ is algebraically closed, $y=Y^p$ for some $Y\in K$. But then if we set $m'=pm$, we see that $c=-Y^{m'}$, that $m'>m$, and that $m'|d$, which contradicts our definition of $m$. Likewise, we claim that $g_\zeta\circ f^n(x)$ is irreducible in $K[x]$ for all $n\geq1$. If not, then \cite[Proposition 2.3]{PreperiodicPointsandABC} implies that $g_\zeta\circ f^n(0)=z^p$ for some $n\geq1$, some prime $p|d$, and some $z\in K$. Then setting $X=f^n(0)$, we have that $X^{r}-\zeta y=z^p$. From here we proceed in cases to show that the prime factors of $d$ are bounded by a constant depending only on the genus of $K$ and $\alpha$. \\[5pt] 
\textbf{Case (1):} Suppose that $r=1$. Now when $n=1$, we see that 
\[-y^d-(y-\alpha)=-y^m-(y-\alpha)=(c-\alpha)-(y-\alpha)=c-y=f(0)-y=g_\zeta\circ f(0)=z^p.\]
On the other hand, if $y-\alpha=0$, then $c-\alpha=-y^m=-y^d=-\alpha^d$ and so $\alpha$ is a fixed point of $f$, a contradiction. In particular, Proposition \ref{prop:FermatCatalan} applied to the equation $-(y-\alpha)=z^p+y^d$ implies that $dh(y)\leq B_1(K)h(y-\alpha)+B_2(K)$ for some $B_1(K)$ and $B_2(K)$ depending on $K$. Moreover, we note that $h(y)>0$, since otherwise $h(c-\alpha)=h(-y^d)=dh(y)=0$, and we again obtain a contradiction. Then $h(y)\geq1$ and so 
\[d\leq B_1(K)(1+h(\alpha))+B_2(K).\] 
Hence, $d$ is bounded by a constant depending only on $K$ and $\alpha$; moreover, $B_1$ is an absolute constant and $B_2$ depends on the genus of $K$. Likewise, when $n\geq2$ we see that  
\[f^{n-1}(0)^d+c-\zeta y=g_\zeta\circ f^n(0)=z^p.\] 
Therefore, if $c-\zeta y\neq 0$, then Proposition \ref{prop:FermatCatalan} and Lemma \ref{lem:orbit0} together imply that \vspace{.1cm}  
\begin{equation*}
\begin{split}
d h(c)\leq dh(f^{n-1}(0))&\leq B_1(K)h(c-\zeta y)+B_2(K)\\[5pt] 
&\leq B_1(K)\big(h(c)+h(y)\big)+B_2(K)\\[5pt] 
&\leq B_1(K)\big(h(c)+dh(y)\big)B_2(K)\\[5pt] 
&\leq B_1(K)\big(h(c)+h(c-\alpha)\big)+B_2(K)\\[5pt]
&\leq B_1(K)\big(2 h(c)+h(\alpha)\big)+B_2(K)
\end{split} 
\end{equation*}
In particular, $d\leq B_1(K)(2+h(\alpha))+B_2(K)$, so that $d$ is bounded by a constant depending on the genus of $K$ and $\alpha$. Finally, if $c-\zeta y=0$, then $c-\alpha=-y^d=-c^d$ since $\zeta^d=1$. Thus $\alpha=c^d+c$ so that 
\[h(\alpha)=h(c^d+c)\geq h(c^d)-h(c)= (d-1)h(c)\geq d-1,\] 
and $d$ is once again bounded by a constant depending on $K$ and $\alpha$. In particular, if we assume that every prime factor of $d$ is sufficiently large depending on $K$ and $\alpha$ (forcing $d$ to be large also), then $g_\zeta\circ f^n(x)$ is irreducible for all $n\geq1$ in this case.   
\\[5pt] 
\textbf{Case (2):} Suppose that $r>1$. Then, if we set $X=f^n(0)$, we have that $X^r-\zeta y=z^p$. Moreover, $h(\zeta y)>0$, since otherwise $h(c-\alpha)=h(-y^m)=0$, a contradicition. Thus, Proposition \ref{prop:FermatCatalan} and Lemma \ref{lem:orbit0} together imply that 
\begin{equation*}
\begin{split} 
r h(c)\leq rh(X)&\leq B_1(K)h(y)+B_2(K) \\[5pt] 
&\leq B_1(K)h(-y^m)+B_2(K) \\[5pt] 
&\leq B_1(K)(h(c)+h(\alpha))+B_2(K).
\end{split}
\end{equation*}
Hence, $r\leq B_1(K)(1+h(\alpha))+B_2(K)$ is bounded by a constant depending on the genus of $K$ and $\alpha$. In particular, if we assume that every prime factor of $d$ is sufficiently large depending on the genus of $K$ and $\alpha$ (forcing $r$ to be large also), then $g_\zeta\circ f^n(x)$ is irreducible for all $n\geq1$ again in this case.

Hence, we deduce that if $d$ is not divisible by $2$ or $3$ and if the prime factors of $d$ are sufficiently large depending on the genus of $K$ and $\alpha$, then the polynomials $g_\zeta\circ f^n$ are irreducible in $K[x]$ for all $n\geq1$. In particular, the number of irreducible factors of $f^n$ in $K[x]$ is equal to the number of irreducible factors of $f$ in $K[x]$; more specifically, all iterates of $f$ have exactly $m\leq d$ factors in $K[x]$.  
\end{proof}
\bibliographystyle{plain}
\bibliography{EventualStability}
\bigskip 
\bigskip

\end{document}